\documentclass[psamsfonts,12pt]{amsart}

\usepackage{enumerate}
\usepackage{amsmath}
\usepackage{color}

\newcommand{\R}{\mathbb{R}}
\newcommand{\N}{\mathbb{N}}

\newcommand{\ci}{\mathop{\mbox{\rm{circ}}}\nolimits}

\newcommand{\vo}{\mathop{\mbox{\rm{vol}}}\nolimits}

\newtheorem{thm}{Theorem}
\newtheorem*{thm*}{Theorem}
\newtheorem{thmp}{Theorem}

\newtheorem{lem}[thm]{Lemma}
\theoremstyle{definition}
\newtheorem{Def}[thm]{Definition}
\newtheorem{exa}{Example}

\begin{document}

%-------preamble---------------------------------------

\title[]{Hermite--Hadamard type inequality for certain Schur convex functions}
\author{P\'al Burai, Judit Mak\'o and Patricia Szokol}
\address{University of Debrecen, 4028, Debrecen, 26 Kassai road, HUNGARY}
\email{burai.pal@inf.unideb.hu}
\address{University of Miskolc, 3515, Miskolc-Egyetemv\'aros, HUNGARY}
\email{matjudit@uni-miskolc.hu}
\address{University of Debrecen, 4028, Debrecen, 26 Kassai road, HUNGARY}
\email{szokol.patricia@inf.unideb.hu}
\maketitle

%-------abstract---------------------------------------

\begin{abstract}
 The main goal of this paper is to prove a Hermite-Hadamard type inequality for certain Schur convex functions using, as one of the main tools in the proof, a Korovkin-type approximation theorem. 
\end{abstract}

%-------body of the text-------------------------------

\section{Introduction}
The classical Hermite--Hadamard inequality \cite{Hadamard1893} is a very important and investigated result concerning convex functions. It states, that if $f\colon [a,b]\to\R$ is convex, then it satisfies the following chain of inequalities
\begin{equation}\label{E:HH}
f\left(\frac{a+b}{2}\right)\le\frac{1}{b-a}\int_a^bf(x)\ dx\le \frac{f(a)+f(b)}{2}.
\end{equation}

In fact, together with continuity, the left-hand side inequality (usually called the lower Her\-mite-Hadamard inequality) implies convexity (see Theorem \ref{T:classical_HH}). Similarly, the upper Hermite-Hadamard inequality together with continuity also implies convexity of $f$.

There are many interesting results concerning this topic. Among others we emphasize that in \cite{Bessenyei2006}, the authors proved an analogous statement concerning convex functions with respect to a Chebyshev system. 

One can find several generalizations of \eqref{E:HH} in the literature, see e.g. \cite{Bessenyei2008a, Bessenyei2003, DragomirAgarwal1998, Niculescu2012, Niculescu2003/04, TsengHwangDragomir2007} and the references therein. 

The main motivation of the present paper comes from \cite{BuraiMako2016}, where the connection between 
a Hermite--Hadamard type inequality and a Jensen type inequality was investigated. 

\begin{thmp}\label{T:BM}
Let $X$ be a normed space, $D\subseteq X$ be a nonempty convex subset of $X$ and $f\colon D\times D\to\R$ be bounded, lower semicontinuous 
and symmetric function. Assume that $\mu$ is a probability measure on 
$[0,1]$, such that $\mu\notin\{\alpha\delta_0+(1-\alpha)\delta_1\mid 
\alpha\in[0,1]\}$. 
Moreover, assume that, for all $x,y\in D,$ the function $f$ satisfies 
the following Hermite--Hadamard type inequality
\begin{equation*}
 \int\limits_{[0,1]}f(tx+(1-t)y,(1-t)y+tx)d\mu(t)\leq f(x,y).
\end{equation*}
Then $f$ is Jensen convex in the following sense:
\[
f\left(\frac{x+y}2,\frac{x+y}2\right)\leq f(x,y), \quad x,y\in D.
\]
\end{thmp}

A simple consequence of this theorem is the following classical result (see \cite{Hadamard1893}).
\begin{thmp}\label{T:classical_HH}
Let $I$ be a proper interval and $f:I\to\R$ be a continuous function on $I$. Then $f$ is convex if and only if
\[
\int_0^1f(tx+(1-t)y)dt\leq \frac{f(x)+f(y)}{2}, \qquad x,y\in I.
\]
\end{thmp}

One of the key tools in the proof of Theorem \ref{T:BM} is the following Korovkin-type theorem from \cite{MakoHazy2017}.  

\begin{thmp}\label{T:HM}
Let $\mathcal{T}_m:\mathcal{C}([0,1])\to \R$ $(m\in\N)$ be a sequence of positive linear operators such that
\[
    \lim_{m\to\infty}(\mathcal{T}_m \textbf{1})=1,
\]
where $\textbf{1}$ stands for the constant 1 function.
Suppose that there exists a function $g\in \mathcal{C}([0,1])$ with $g\left(\frac12\right)=0$ and $g>0$ on $[0,1]\setminus \left\{\frac12\right\}$ such that $\lim\limits_{m\to\infty}(\mathcal{T}_mg)=0$. Then, for all bounded lower semicontinuous function $h\colon[0,1]\to\R$,
\[
    \lim_{m\to\infty}\mathcal{T}_m h = h\left(\tfrac12\right).
\]
\end{thmp}

We recall, the statement of the classical Korovkin theorem. Let $I$ be a bounded, closed interval and $L_m\colon \mathcal{C}(I) \to \mathcal{C}(I),\ m\in\N$ be a sequence of positive, linear mappings. If $L_mf\to f$ in supremum norm for the functions $\textbf{1}, x, x^2$, then $L_mf\to f$ for all $f\in\mathcal{C}(I)$. 

As a matter of fact, in \cite{Korovkin1953} Korovkin obtained a result, which has a consequence similar to the previous one. More precisely, if the sequence $L_m$ defined on the set of all continuous functions that are periodic with period $2\pi$, and $\|f-L_mf\|_{\infty}\to 0$ for the constant 1, cosine and sine functions, then $\|f-L_mf\|_{\infty}\to 0$ for every periodic function $f$. For more details see e.g. \cite{Bohman1952, Komornik2016, Korovkin1953}.

One of the most important consequences of the classical Korovkin theorem is the well-known first approximation theorem of Weierstrass, which says that all continuous functions defined on a compact interval can be approximated uniformly by polynomials. Besides that fundamental theorem of approximation theory, there are many other variants of Korovkin type theorems in the literature. For more different versions, generalizations, and applications we refer to the papers \cite{Altomare1994, Korovkin1960, MakoPales2012, MakoHazy2017} and the references therein.

In the main theorem of this paper we are going to prove a result, which can be considered as a multivariable case of Theorem \ref{T:BM}. More precisely, we will investigate symmetric, continuous functions whose domain is in $\R^n$ that satisfies a Hermite--Hadamard type inequality. Our aim is to confirm that such functions are necessarily Jensen-convex. 

One of the main tools, just as in the case of two-variables functions, is  a Korovkin-type theorem (Theorem \ref{T:Korovkin}). However, we need also some lemmas, and the proof of the main result is quite lengthy and more difficult from technical point of view than the proof of Theorem \ref{T:BM}.

Our work is organized as follows. Section 2 recalls some known results, which will be important for our later purposes, and introduces notations and terminology used throughout this paper. Section 3 contains the main result and a Korovkin-type theorem together with its proof. Section 4 is devoted to the proof of the main theorem. Finally, Section 5 presents some applications.

\section{Notations and basic results}
 
Let $n\in\N$ be a fixed natural number, $c\in \R^n$ be an arbitrary vector and $(c_m)$ be a sequence in $\R^n$. The $i$th coordinates of $c$ and $c_m$ are denoted by $c^i$ and $c_m^i$, respectively, i.e.
\[
c=\begin{bmatrix}
c^1,&\dots,&c^n
\end{bmatrix},\quad\mbox{and}\qquad
c_m=\begin{bmatrix}
c^1_m,&\dots,&c^n_m
\end{bmatrix}.
\]
An $n\times n$ square matrix is said to be \textit{doubly stochastic} if its elements are all non-negative and all row and column sums are one.

Furthermore, we define a special class of doubly stochastic matrices. For this, let us recall the definition of simplices in $\R^n$. Let $S_n(a^1,\ldots,a^n)=S_n(a)$ be an arbitrary simplex in $\R^n$, where $a^1,\ldots, a^n,$ $a^j>0,$ $j=1,\ldots n$ and
\begin{eqnarray*}
S_n(a)=\left\{c\in\R^n\ \Big|\ \frac{c^1}{a^1}+\ldots +\frac{c^n}{a^n}=1,\ c^j\ge 0,\ j=1,\ldots, n\right\}.
\end{eqnarray*}
The volume of $S_n(a^1,\ldots,a^n)$ is (see e.g. \cite{Bessenyei2008} or \cite{Ellis1976})
\begin{eqnarray}\label{E:volsimplex}
\vo S_n(a)=\int\limits_{0}^{a^1}\ldots\int\limits_0^{a^n\left(1-\frac{c^1}{a^1}-\ldots-\frac{c^{n-1}}{a^{n-1}}\right)}1\ dc^n\ldots dc^1=\frac{\prod_{i=1}^n a^i}{n!}.
\end{eqnarray}

For the sake of simplicity, we denote the standard simplex in $\R^n$, i.e. the simplex with unit sides by $S_n$. In particular, the equation \eqref{E:volsimplex} gives 
\[
\vo S_n=\vo S_n(1,\ldots,1)=\frac{1}{n!}.
\]  
We say, that a matrix is circulant doubly stochastic (see e.g. \cite{Davis1979}) generated by $c=(c^1,\ldots,c^n)\in S_n$ if it has the following form
\[
C=\ci(c)=\ci\begin{bmatrix}
c^1,&\dots,&c^{n}
\end{bmatrix}=\begin{bmatrix}
c^1&c^2&\dots&c^{n-1}&c^{n}\\
c^{n}&c^1&\cdots&c^{n-2}&c^{n-1}\\
\vdots&\vdots&\cdots&\vdots&\vdots\\
c^2&c^3&\dots&c^{n}&c^{1}
\end{bmatrix}.
\]

%------------------------------------------------
%------------------------------------------------

\section{Main results}

Our main result shows that a certain Hermite--Hadamard type inequality implies a Jensen-type inequality .
\begin{thm}\label{T:main_theorem}
Let $D$ be a nonempty, convex subset of $\R^n$ and $f\colon D\to\R$ be a symmetric, continuous function for which
\begin{equation}\label{E:main_result_assumption}
\frac{1}{\vo S_n}\int\limits_{S_n} f(\ci(c)x)\ d\lambda^n(c)\leq f(x),\qquad x\in D,
\end{equation}
where $\lambda^n$ denotes the Lebesgue measure on $\R^n$.
Then
\begin{equation}\label{E:main_result_statement}
f\left(\frac{x^1+\cdots+x^n}{n}, \dots, \frac{x^1+\cdots+x^n}{n}\right)\leq f(x^1,\dots,x^n),
\end{equation}
for all $(x^1,\dots,x^n)\in D$.
\end{thm}

Its proof is based on our second main theorem, which can be considered as a Korovkin-type result. 

For the readers convenience, we recall the definition of monotone and positive operators and the relationship between them.

\begin{Def}
Let $K$ be a compact subset of $\R^n$ and $\mathcal{C}(K)$ denote the set of all continuous real-valued functions on $K$. 
We say, that a functional $\mathcal{T}\colon \mathcal{C}(K)\to \R$ is
\begin{itemize}
\item[(a)] monotone, if $f(x)\le g(x)$, $(x\in K)$ implies that $\mathcal{T}(f)\le \mathcal{T}(g)$;
\item[(b)] positive, if $f(x)\ge 0$, $(x\in K)$ implies that $\mathcal{T}(f)\ge 0$. 
\end{itemize}
\end{Def} 

By the previous definition it is easy to see, that every positive, linear functional is monotone. Indeed, let $f, g\in \mathcal{C}(K)$, such that $f(x)\le g(x)$, $(x\in K)$. Then, $0\le g(x)-f(x)$ $(x\in K)$ and hence $0\le \mathcal{T}(g-f)$. By the linearity of $\mathcal{T}$ we infer $\mathcal{T}(f)\le \mathcal{T}(g)$.  

The reverse implication is trivially not true.

We can formulate now our second main theorem.
\begin{thm}\label{T:Korovkin}
Let $K$ be a compact subset of  $\ \R^n$ with non-empty interior ($K^{\circ}\not=\emptyset$), $\mathcal{T}_m\colon\mathcal{C}(K)\to\R$ be a sequence of positive, linear functionals for which 
\begin{equation}\label{E:operatorsorozat_limesz}
\lim_{m\to\infty}\mathcal{T}_m \textbf{1}=1.
\end{equation}
Moreover, let $p\in K^{\circ},\ g\in \mathcal{C}(K)$ such that
\begin{equation}\label{E:feltetelek_g_re}
g(p)=0,\qquad\mbox{and}\qquad g>0\quad\mbox{on}\quad K\setminus\{p\},
\end{equation}
and
\begin{equation}\label{E:limesz_feltetel_g_re}
\lim_{m\to\infty}\mathcal{T}_mg=0=g(p).
\end{equation}
Then
\begin{equation}\label{E:allitas}
\lim_{m\to\infty}\mathcal{T}_mh=h(p),\qquad h\in\mathcal{C}(K).
\end{equation}
\end{thm}

\begin{proof}
Let $h\in\mathcal{C}(K)$, and $\varepsilon>0$ be arbitrary. Let us define the following function
\[
\varphi(x)=\frac{\varepsilon-|h(x)-h(p)|}{g(x)},\qquad x\in K\setminus\{p\}.
\]
Because of \eqref{E:feltetelek_g_re} $\varphi$ is well-defined.

By the continuity of $h$, there is a $\delta$ such that 
\[
|h(x)-h(p)|<\varepsilon\quad\mbox{if}\quad\|x-p\|<\delta.
\]
So, $\varphi>0$ on $B(p,\delta)\setminus\{p\}$, where $B(p,\delta)$ denotes the open ball with center $p$ and with radius $\delta$. 

The function $\varphi$ is continuous on the compact set $K\setminus B(p,\delta)$.  Let us denote by $L$ the minimum of $\varphi$ on $K\setminus B(p,\delta)$ and let $\tilde{L}$ be the minimum of $L$ and zero. Then
\[
\varphi>\tilde{L}\qquad\mbox{on}\qquad K\setminus\{p\}.
\]
Using this latter inequality we have
\[
\frac{\varepsilon-|h(x)-h(p)|}{g(x)}>\tilde{L}\qquad\mbox{on}\qquad K\setminus\{p\}.
\]
This implies
\[
-\varepsilon+\tilde{L}g(x)<h(x)-h(p)<\varepsilon-\tilde{L}g(x).
\]
Applying $\mathcal{T}_m$ on the previous trail of inequalities, using its monotonicity and linearity, we get
\[
-\varepsilon\mathcal{T}_m \textbf{1}+\tilde{L}\mathcal{T}_mg<\mathcal{T}_m h-(\mathcal{T}_m \textbf{1}) h(p)<\varepsilon\mathcal{T}_m \textbf{1}-\tilde{L}\mathcal{T}_mg.
\]  
Taking the limits, because of \eqref{E:operatorsorozat_limesz} and \eqref{E:limesz_feltetel_g_re} we can derive the following inequality
\[
-\varepsilon\leq \mathcal{T}_m h-h(p)\leq \varepsilon,
\]
which entails \eqref{E:allitas}.
\end{proof}

\section{Proof of the Main Theorem}

The proof of Theorem \ref{T:main_theorem} is quite lengthy, and besides of Theorem \ref{T:Korovkin}, some technical lemmas are also needed. This is why we devote a full section to its proof.

\begin{lem}\label{L:sok_integral}
Let $m$ be an arbitrary natural number, $c_1,\dots,c_m\in S_n$. Besides of the assumption of Theorem \ref{T:main_theorem} we have
\begin{equation*}
\frac{1}{(\vo S_n)^m}\int\limits_{S_n}\cdots\int\limits_{S_n} f(\ci(c_1)\cdots\ci(c_m)x)\ d\lambda^n(c_1)\cdots d\lambda^n(c_m)\leq f(x)
\end{equation*}
for every $x\in D$.
\end{lem}
\begin{proof}
It comes from \eqref{E:main_result_assumption} using induction with respect to $m$.
\end{proof}

\begin{lem}\label{L:ciklikus_szorzata}
Let $m$ be an arbitrary natural number, $c_1,\dots,c_m\in S_n$ and
\begin{equation*}
T_m=\ci(c_1)\dots\ci(c_m)=\ci(t_m),
\end{equation*}
then for all $m\ge 2$
\begin{align}\label{E:ciklikus}
t_m^\alpha\!=\!\frac{1}{n}\!+\!\!\!\sum_{i_{m-1}=1}^n\!\dots\!\!\sum_{i_{1}=1}^n\!\!\left(\!c_1^{i_1}-\frac{1}{n}\!\right)\!\!\left(\!c_2^{i_2-i_1+1}-\frac{1}{n}\!\right)\!\cdots\!\left(\!c_m^{\alpha-i_{m-1}+1}-\frac{1}{n}\!\right),
\end{align}
where $\alpha=1,\dots,n$.
\end{lem}
\begin{proof}
We prove the lemma by induction with respect to $m$.  
Let $m=2$, and $T_2=\ci(c_1)\ci(c_2)=\ci(t_2)$.
Applying the circulant property of $\ci(c_2)$ and the fact $\sum_{i=1}^nc_1^i=\sum_{i=1}^nc_2^i=1$, we get that
\begin{align*}
t_2^\alpha&=\sum_{i=1}^n c_1^ic_2^{\alpha-i+1}=\sum_{i=1}^n\left[\frac{1}{n}+\left(c_1^i-\frac{1}{n}\right)\right]\left[\frac{1}{n}+\left(c_2^{\alpha-i+1}-\frac{1}{n}\right)\right]\\
&=\sum_{i=1}^n \left[\frac{1}{n^2}+\frac{1}{n}\left(c_1^i-\frac{1}{n}+c_2^{\alpha-i+1}-\frac{1}{n}\right)+\left(c_1^i-\frac{1}{n}\right)\left(c_2^{\alpha-i+1}-\frac{1}{n}\right)\right]\\
&=\frac{1}{n}+\sum_{i=1}^n \left(c_1^i-\frac{1}{n}\right)\left(c_2^{\alpha-i+1}-\frac{1}{n}\right), \quad \alpha=1,\ldots,n,
\end{align*}
where the index $\alpha$ is taken modulo $n$.

Assume now that \eqref{E:ciklikus} holds for $m=k$ and we are going to show that it holds for $m=k+1$, as well. Let $T_{k+1}=\ci(c_1)\ldots\ci(c_k)\ci(c_{k+1})$. Then, 
\[
t_{k+1}^\alpha=\sum_{i_k=1}^n t_k^{i_k}c_{k+1}^{\alpha-i_k+1},
\]
where by the induction hypothesis we know that for all $\alpha=1,\ldots,n$
\[
t_k^\alpha=\frac{1}{n}+\sum_{i_{k-1}=1}^n\ldots \sum_{i_1=1}^n\left(c_1^{i_1}-\frac{1}{n}\right)\left(c_2^{i_2-i_1+1}-\frac{1}{n}\right)\cdots\left(c_k^{\alpha-i_{k-1}+1}-\frac{1}{n}\right).
\]
Then, 
\begin{align*}
&t_{k+1}^\alpha=\sum_{i_k=1}^n \left[\frac{1}{n}+\sum_{i_{k-1}=1}^n\ldots \sum_{i_1=1}^n\left(c_1^{i_1}-\frac{1}{n}\right)\cdots\left(c_k^{i_k-i_{k-1}+1}-\frac{1}{n}\right)\right]\\
&\times\left[\frac{1}{n}+\left(c_{k+1}^{\alpha-i_k+1}-\frac{1}{n}\right)\right]\\
&=\frac{1}{n}+\frac{1}{n}\sum_{i_k=1}^n\ldots\sum_{i_1=1}^n\left(c_1^{i_1}-\frac{1}{n}\right)\ldots\left(c_{k}^{i_k-i_{k-1}+1}-\frac{1}{n}\right)\\
&+\frac{1}{n}\sum_{i_k=1}^n\left(c_{k+1}^{\alpha-i_k+1}-\frac{1}{n}\right)+\sum_{i_k=1}^n\ldots\sum_{i_1=1}^n\left(c_1^{i_1}-\frac{1}{n}\right)\ldots\left(c_{k+1}^{\alpha-i_k+1}-\frac{1}{n}\right).
\end{align*}
It can be shown that the second and third terms of the previous expression are equal to 0. Indeed, using that $\sum_{i=1}^nc_k^i=1$ we get that
\[
\sum_{i_k=1}^n\left(c_{k}^{i_k-i_{k-1}+1}-\frac{1}{n}\right)=0
\]
and hence, concerning the second term
\begin{align*}
&\frac{1}{n}\sum_{i_k=1}^n\ldots\sum_{i_1=1}^n\left(c_1^{i_1}-\frac{1}{n}\right)\ldots\left(c_{k}^{i_k-i_{k-1}+1}-\frac{1}{n}\right)\\
=&\frac{1}{n}\sum_{i_{k-1}=1}^n\ldots\sum_{i_1=1}^n\left(c_1^{i_1}-\frac{1}{n}\right)\ldots\sum_{i_k=1}^n\left(c_{k}^{i_k-i_{k-1}+1}-\frac{1}{n}\right)=0.
\end{align*}
Similarly, $\sum_{i=1}^nc_{k+1}^i=1$, which implies that the third term is also equal to 0. Consequently, we get that the desired formula for $t_{k+1}^{\alpha}$.
\end{proof}

In the proof of Theorem \ref{T:main_theorem} we need to have an explicit formula of the following two integrals.
Let $s$ be an arbitrary positive real number and $I_n(s)$, $J_n(s)$ denote the integrals 
\[
I_n(s):=\int\limits_0^s \int\limits_0^{1-c^1} \ldots \int\limits_0^{1-c^1-\ldots -c^{n-1}} \textbf{1} \, d\lambda(c^n)\ldots d\lambda(c^2) d\lambda(c^1);
\]
and 
\[
J_n(s):=\int\limits_0^s\int\limits_0^{1-c^1}\ldots\int\limits_0^{1-c^1-\ldots-c^{n-1}} c^1 \,  d\lambda(c^n)\ldots d\lambda(c^2) d\lambda(c^1),
\]
respectively.
\begin{lem}\label{L:integralok}
Let $s\ge 0$ be an arbitrary real number. Then, 
\begin{equation}\label{E:int1}
I_n(s)=-\sum_{i=1}^n \frac{1}{i!(n-i)!}(-s)^i, \quad n\in \N;
\end{equation}
and
\begin{equation}\label{E:int2}
J_n(s)=\sum_{i=1}^n \frac{i}{(i+1)!(n-i)!}(-s)^{i+1}, \quad n\in \N.
\end{equation}
\end{lem}
Besides the integrals over a simplex $I_n(s)$ and $J_n(s)$, we recall an other nice expression, which will be used in the proof of Theorem \ref{T:main_theorem}. 
Let $s$ be an arbitrary positive real number, then
\begin{equation}\label{E:int3}
K_n(s):=\int\limits_0^s\int\limits_0^{s-c^1}\ldots\int\limits_0^{s-c^1-\ldots-c^n} c^1\  d\lambda(c^n)\ldots d\lambda(c^2) d\lambda(c^1)=\frac{s^{n+1}}{(n+1)!}.
\end{equation}
For more details see e.g \cite{Sobczyk1992}.

Now, we are in a position to prove Lemma \ref{L:integralok}.
\begin{proof}
 It is easy to see, that the integral $I_n(s)$ can be formulated 
 in the following way
\begin{align*}
&I_n(s)=\\
&\int\limits_0^s\left(\int\limits_0^{1-c^1}\, \int\limits_0^{(1-c^1)\left(1-\frac{c^2}{1-c^1}\right)}\ldots\int\limits_0^{(1-c^1)\left(1-\sum\limits_{j=1}^{n-1}\frac{c^j}{1-c^1}\right)} \textbf{1} \, d\lambda(c^n)\ldots d\lambda(c^2)\right) d\lambda(c^1),
\end{align*}
where the integral appearing in the brackets is the volume of the simplex $S_{n-1}(1-c^1,\ldots,1-c^1)$. By \eqref{E:volsimplex} we get 
\[
\vo(S_{n-1}(1-c^1,\ldots,1-c^1))=\frac{(1-c^1)^{n-1}}{(n-1)!}.
\]
Consequently,
\begin{align*}
I_n(s)&=\frac{1}{(n-1)!}\int\limits_0^s (1-c^1)^{n-1} \, d\lambda(c^1)= -\frac{1}{n(n-1)!}\left[(1-c^1)^n\right]_{0}^{s}\\
&= -\frac{1}{n!}\left(\sum_{i=0}^n {n\choose i} (-s)^i-1\right)=-\sum_{i=1}^n \frac{1}{i!(n-i)!}(-s)^i.
\end{align*}

Just as in the case of $I_n(s)$, one can rewrite the integral $J_n(s)$ and apply \eqref{E:volsimplex}
\begin{align*}
&J_n(s)=\int\limits_0^s \int\limits_0^{1-c^1} \ldots\int\limits_0^{1-c^1-c^2-\ldots-c^{n-1}} c^1 \,  d\lambda(c^n)\ldots d\lambda(c^2) d\lambda(c^1)=\\
%%%%%%%%%%%%&=\int\limits_0^s c^1 \left(\int\limits_0^{1-c^1}\int\limits_0^{(1-c^1)\left(1-\frac{c^2}{1-c^1}\right)}\dots\int\limits_0^{(1-c^1)\left(1-\frac{c^2}{1-c^1}-\dots-\frac{c^{n-1}}{1-c^1}\right)}\ dc^n\dots dc^3dc^2 \right)dc^1\\
& \int\limits_0^s \! c^1 S_{n-1}(1-c^1,\ldots,1-c^1) d\lambda(c^1)= \frac{1}{(n-1)!}\int\limits_0^s \!c^1(1-c^1)^{n-1} d\lambda(c^1). 
\end{align*} 
Using the binomial theorem for the expression $(1-c^1)^{n-1}$, we get that
\begin{equation*}
c^1\sum_{j=0}^{n-1}\! {n-1 \choose j}\!(-c^1)^j=-\sum_{j=0}^{n-1}\! {n-1 \choose j}\!(-c^1)^{j+1}=-\sum_{i=1}^{n}\! {n-1 \choose i-1}(-c^1)^i.
\end{equation*}
It implies, that 
\begin{align*}
J_n(s)&=-\frac{1}{(n-1)!} \sum_{i=1}^n \left({n-1\choose i-1}\int\limits_0^s (-c^1)^i \, dc^1\right)\\
&=\sum_{i=1}^n \frac{i}{(i+1)!(n-i)!}(-s)^{i+1}, 
\end{align*}
which completes the proof.
\end{proof}

Now, we are going to present the proof of the Main Theorem \ref{T:main_theorem}.

\begin{proof}[Proof of Theorem \ref{T:main_theorem}]
In the proof we are going to apply the Korovkin-type Theorem \ref{T:Korovkin}. To do this we need to define a sequence of operators $\mathcal{T}_m\colon\mathcal{C}(S_n)\to\R$, $m\in\N$, which satisfies the assumptions appearing in Theorem \ref{T:Korovkin}. Let $h\colon S_n\to \R$.
Then the operators $\mathcal{T}_m$, $m\in \N$ defined by  
\[
\mathcal{T}_m(h)=\frac{1}{(\vo S_n)^m}\int\limits_{S_n}\ldots\int\limits_{S_n} h(\ci(c_1)\cdots\ci(c_m))\ d\lambda^n(c_1)\cdots d\lambda^n(c_m)
\]
are positive, linear on $\mathcal{C}(S_n)$. Moreover,
\[
\mathcal{T}_m(\textbf{1})=\frac{1}{(\vo S_n)^m}\int\limits_{S_n}\cdots\int\limits_{S_n} \textbf{1} \ d\lambda^n(c_1)\cdots d\lambda^n(c_m)=\frac{(\vo S_n)^m}{(\vo S_n)^m}=1,
\]
for all $m\in\N$, i.e. it also satisfies \eqref{E:operatorsorozat_limesz}.

Now, we are going to introduce the function $g\colon S_n\to \R$ defined by 
\begin{equation*}
g(t^1,\ldots,t^n)=\sum_{j=1}^n \left|t^j-\frac{1}{n}\right|, \quad (t^1,\ldots,t^n)\in S_n.
\end{equation*}
Then it is easy to see that $g$ vanishes at $\left(\frac{1}{n},\ldots, \frac{1}{n}\right)$ and $g$ is strictly positive elsewhere. 
To apply Theorem \ref{T:Korovkin} we need to show that $g$ satisfies the assumption \eqref{E:limesz_feltetel_g_re}, that is 
\begin{equation}
\lim_{m\to\infty}\mathcal{T}_mg=0=g\left(\frac{1}{n},\ldots, \frac{1}{n}\right),
\end{equation}
or equivalently
\begin{equation}\label{E:t-limesz}
\lim_{m\to \infty}\frac{1}{(\vo S_n)^m}\int\limits_{S_n}\ldots\int\limits_{S_n} g(t_m^1,\ldots,t_m^n)\ d\lambda^n(c_1)\ldots\lambda^n(c_m)=0,
\end{equation}
where $t_m$ is the first row of $\ci(c_1)\cdots\ci(c_m)$, i.e. 
\[
\ci(c_1)\cdots\ci(c_m)=\ci(t_m).
\]
We are going to apply the squeeze theorem. On the one hand, the above mentioned integral is positive since we take the integral of a nonnegative function over a simplex. On the other hand we verify that it can be majorized by a null sequence. More precisely,
\begin{equation}\label{E:squeeze}
\frac{1}{(\vo S_n)^m}\int\limits_{S_n}\ldots\int\limits_{S_n} g(t_m^1,\ldots,t_m^n)\ d\lambda^n(c_1)\ldots\lambda^n(c_m)\le \left(\frac{n \mathcal{I}_n}{\vo S_n}\right)^m
\end{equation}
where 
\[
\mathcal{I}_n=\int\limits_{S_n}\left|c^1-\frac{1}{n}\right|\ d\lambda^n(c), \quad c\in S_n.
\]
Moreover, the right-hand side of \eqref{E:squeeze} tends to 0, if $m\to \infty$.

Firstly, we show that the inequality \eqref{E:squeeze} holds. 

By Lemma \ref{L:ciklikus_szorzata} we know that 
\[
t_m^\alpha=\frac{1}{n}+\sum_{i_{m-1}=1}^n \dots\sum_{i_{1}=1}^n\left(c_1^{i_1}-\frac{1}{n}\right)\!\!\left(c_2^{i_2-i_1+1}-\frac{1}{n}\right)\cdots\left(c_m^{\alpha-i_{m-1}+1}-\frac{1}{n}\right).
\]
Then, 
\begin{align*}
g(t_m^1,\ldots,t_m^n)&=\sum_{\alpha=1}^n \left|\sum_{i_{m-1}=1}^n \dots\sum_{i_{1}=1}^n\left(c_1^{i_1}-\frac{1}{n}\right)\cdots\left(c_m^{\alpha-i_{m-1}+1}-\frac{1}{n}\right)\right|\\
&\le\sum_{\alpha=1}^n \sum_{i_{m-1}=1}^n \dots\sum_{i_{1}=1}^n\left|\left(c_1^{i_1}-\frac{1}{n}\right)\cdots\left(c_m^{\alpha-i_{m-1}+1}-\frac{1}{n}\right)\right|
\end{align*}
which means that there are $n\cdot n^{m-1}=n^m$ products in which all terms $c_j^i$ are pairwise independent $(j=1,\ldots,m,$ $i=1,\ldots, n)$. Moreover, if $c\in S_n$ is arbitrary, then, it is easy to see, that
\[
\int\limits_{S_n}\left|c_j^i-\frac{1}{n}\right|d\lambda^n(c_j)=\int\limits_{S_n}\left|c^1-\frac{1}{n}\right|d\lambda^n(c)=\mathcal{I}_n
\] 
for any $j=1,\ldots,m$, $i=1,\ldots, n$.
Consequently, 
\begin{align*}
\frac{1}{(\vo S_n)^m}\int\limits_{S_n}\ldots\int\limits_{S_n} g(t_m^1,\ldots,t_m^n)\ d\lambda^n(c_1)\ldots\lambda^n(c_m)\le
\frac{n^m\mathcal{I}_n^m}{\left(\vo S_n \right)^m}.
\end{align*}

Now, we prove that the right-hand side of \eqref{E:squeeze} tends to 0, if $m\to \infty$. Applying Lemma \ref{L:integralok} for $s=\frac{1}{n}$, one can calculate the value of $\mathcal{I}_n$. Let $c\in S_n$. Then, 
\begin{align*}
\mathcal{I}_n&=\int\limits_0^1 \int\limits_0^{1-c^1}\ldots \int\limits_0^{1-c^1-\ldots -c^{n-1}} \left|c^1-\frac{1}{n}\right|\ d\lambda^n(c)\\
&=\int\limits_0^{\frac{1}{n}}\int\limits_0^{1-c^1}\ldots\int\limits_0^{1-c^1-\ldots-c^{n-1}}\left(-c^1+\frac{1}{n}\right)\ d\lambda(c^n)\ldots d\lambda(c^2) d\lambda(c^1)\\
&+\int\limits_{\frac{1}{n}}^1\int\limits_0^{1-c^1}\ldots\int\limits_0^{1-c^1-\ldots-c^{n-1}}\left(c^1-\frac{1}{n}\right)\ d\lambda(c^n)\ldots d\lambda(c^2) d\lambda(c^1).
\end{align*}

The first integral of the right-hand side can separated into two parts and using the notation $I_n(s), J_n(s)$, we get that it equals to $-J_n(1/n)+1/nI_n(1/n)$. Concerning the second integral we introduce a new variable. Let $b=c^1-\frac{1}{n}$, then the second part of $\mathcal{I}_n$ equals to
\begin{align*}
\int\limits_0^{1-\frac{1}{n}}\int\limits_0^{\left(1-\frac{1}{n}\right)-b}\ldots\int\limits_0^{\left(1-\frac{1}{n}\right)-b-\ldots-c^{n-1}} b \ d\lambda(c^n)\ldots d\lambda(c^2) d\lambda(b)=K_n\left(1-\frac{1}{n}\right).
\end{align*}
Consequently, applying Lemma \ref{L:integralok} and \eqref{E:int3} we get that
\begin{align*}
\mathcal{I}_n&=-J_n\left(\frac{1}{n}\right)+\frac{1}{n}I_n\left(\frac{1}{n}\right)+K_n\left(1-\frac{1}{n}\right)\\
&=\frac{1}{n}\left[\sum_{i=1}^n \frac{1}{i!(n-i)!}\left(-\frac{1}{n}\right)^{i}\left(\frac{i}{i+1}-1\right)\right]+\frac{\left(1-\frac{1}{n}\right)^{n+1}}{(n+1)!}.
\end{align*}
Now, we are in a position to prove that 
\[
\frac{n\mathcal{I}_n}{\vo S_n}=n\cdot n!\cdot \mathcal{I}_n<1, \quad \forall n\in\N.
\]
Easy calculation shows that
\begin{align*}
&n\cdot n!\cdot \mathcal{I}_n=n!\left[\sum_{i=1}^n \frac{1}{i!(n-i)!}\left(-\frac{1}{n}\right)^{i}\left(\frac{i}{i+1}-1\right)+\frac{\left(n-1\right)^{n+1}}{n^n(n+1)!}\right]\\
&=\sum_{i=1}^n \frac{n}{i!(n-i)!}\left(-\frac{1}{n}\right)^{i}\left(\frac{i}{i+1}-1\right)+\frac{\left(n-1\right)^{n+1}}{n^n(n+1)}.
\end{align*}
Separating the first sum into two parts and applying the binomial theorem we infer that
\[ 
n\cdot n!\cdot \mathcal{I}_n=\sum_{i=1}^n {n\choose i}\frac{i}{i+1}\left(-\frac{1}{n}\right)^{i}-\left(1-\frac{1}{n}\right)^n+1+\left(1-\frac{1}{n}\right)^n\frac{n-1}{n+1}.
\]
To prove that the previous expression is strictly less than 1, it is enough to show that $n\cdot n!\cdot\mathcal{I}_n-1<0$, i.e.
\[
-\left(1-\frac{1}{n}\right)^n\cdot \frac{2}{n+1}+\sum_{i=1}^n {n\choose i}\frac{i}{i+1}\left(-\frac{1}{n}\right)^{i}<0.
\]
The first term is zero for $n=1$ and negative for $n>1$. Therefore it is enough to show that the second term is negative for $n=1$ and non-positive for $n>1$. 

 The second term starts with $-\frac{1}{2}$ and it consists of alternating terms which are monotone decreasing in absolute value. Indeed, let $i$ be fixed $i=1,\ldots,n-1$. Then, we assert that
\[
\left|\frac{i}{i+1}{n\choose i}\left(-\frac{1}{n}\right)^i\right|>\left|\frac{i+1}{i+2}{n\choose {i+1}}\left(-\frac{1}{n}\right)^{i+1}\right|.
\]
Independently of the parity of $i$, the previous inequality is equivalent to the following one
\[
\frac{i}{i+1}{n\choose i}\frac{1}{n^i}>\frac{i+1}{i+2}{n\choose {i+1}}\frac{1}{n^{i+1}},
\]
which, after simplification, becomes
\[
1>\frac{i+1}{i+2}\cdot\frac{n-i}{n}\cdot\frac{1}{i}. 
\]
The above inequality is trivially true, so is the following one.
\[
\frac{n\mathcal{I}_n}{\vo S_n}<1,
\]
hence
\[
\left(\frac{n\mathcal{I}_n}{\vo S_n}\right)^m\to\infty,\qquad\mbox{as}\quad m\to\infty.
\] 

Consequently, we get that the assumption  \eqref{E:limesz_feltetel_g_re} is also satisfied. 
Applying Theorem \ref{T:Korovkin}, we infer  
\begin{equation}\label{E:lim_assumption}
\lim_{m\to \infty} \mathcal{T}_m h=h\left(\frac{1}{n},\ldots,\frac{1}{n} \right), \quad h\in \mathcal{C}(S_n).
\end{equation}
Let $x=\left(x^1,\ldots,x^n\right)$ be a fixed element of $D$ and $h\colon S_n\to \R$ be defined by 
\[
h(c)=f(\ci(c)x), \quad c\in S_n.
\]
Then, $h\in\mathcal{C}(S_n)$, hence \eqref{E:lim_assumption} holds for $h$, that is to say,
\begin{align*}
\lim_{m\to \infty}\mathcal{T}_m(h)&=
h\left(\frac{1}{n},\ldots,\frac{1}{n} \right)=f\left(\ci\left(\frac{1}{n},\ldots,\frac{1}{n} \right)x\right)\\\\
&=f\left(\frac{x^1+\cdots+x^n}{n},\dots,\frac{x^1+\cdots+x^n}{n}\right)
\end{align*}

Finally, by Lemma \ref{L:sok_integral} we also get that
\begin{align*}
\lim_{m\to \infty}\mathcal{T}_m(h)\le f(x^1,\ldots,x^n),
\end{align*}
which completes the proof.
\end{proof}

%------------------------------------------------
%------------------------------------------------

\section{Applications and special cases}

Convexity, quasi-convexity and strong convexity have a significant role in optimization theory. Their use is based on the local-global minimum property (see e.g. \cite{Burai2014}, \cite{Burai2013} and the references therein). More precisely, all local minimizers of the previously listed functions are global minimizers as well. 

We deal with only the one-variable case here. The interested reader can derive higher dimensional examples in a pretty similar way.
\subsection{Proving convexity with Theorem \ref{T:main_theorem}}\label{S:proving_convexity}

Let $I\subset\R$ be an open interval, and $g\colon I\to\R$ be a continuous function. If the function $f(x^1,x^2)=g(x^1)+g(x^2)$ fulfils inequality \eqref{E:main_result_assumption}, then, according to Theorem \ref{T:main_theorem}, it also fulfils inequality \eqref{E:main_result_statement}, which takes the following shape:
\[
2g\left(\frac{x^1+x^2}{2}\right)\leq g(x^1)+g(x^2),\qquad x^1,x^2\in I,
\]
which means $g$ is Jensen convex. This implies its convexity together with continuity (see e.g. \cite[Theorem 7.1.1]{Kuczma2009}).

So, we have to check the fulfilment of inequality \eqref{E:main_result_assumption}, that is to say:
\[
\frac{1}{2}\int\limits_0^1\int\limits_0^{1-t}g(sx^1+tx^2)+g(tx^1+sx^2)\ ds\, dt\leq g(x^1)+g(x^2).
\]

In fact, by changing variables we obtain that this inequality is equivalent to the following one
 \[
\int\limits_0^1\int\limits_0^{1-t}g(sx^1+tx^2)\ ds\, dt\leq g(x^1)+g(x^2).
\]

\begin{exa}
Let $g(x)=x^2$, $x\in\R$. Then, we have to show that
\begin{align*}
\int\limits_0^1\int\limits_0^{1-t} \left(sx^1+tx^2\right)^2 \ ds\, dt \leq (x^1)^2+(x^2)^2, \quad x^1,x^2\in \R.
\end{align*}
Calculating the double integral, the previous inequality implies that  
\begin{align*}
\frac{1}{12}\left((x^1)^2+(x^2)^2+x^1x^2\right) \leq (x^1)^2+(x^2)^2,
\end{align*} 
or equivalently, 
\begin{align*}
0\leq (x^1)^2+(x^2)^2-\frac{1}{12}\left((x^1)^2+(x^2)^2+x^1x^2\right), \quad x^1,x^2\in \R.
\end{align*}
To take the complete square of the expression appearing on the right-hand side with respect to $x^1$, we get 
\begin{align*}
\frac{11}{12}\left(x^1-\frac{1}{22}x^2\right)^2+\frac{161}{76}(x^2)^2,
\end{align*}
which is nonnegative. Consequently, the inequality assumption \eqref{E:main_result_assumption} holds, which entails convexity of $g$.
\end{exa}

\begin{exa}
Let us consider the function $g(x)=\frac{1}{x},\ x>0$. Assuming that $0<x^1\le x^2$ we have
\begin{align*}
\int\limits_0^1\!\int\limits_0^{1-t}\frac{1}{sx^1+tx^2}\,ds\,dt\leq
\int\limits_0^1\int\limits_0^{1-t}\frac{1}{(s+t)x^1}\ ds\, dt\le -\frac{1}{x^1} \int\limits_0^1 \ln(t)\ dt =\\\frac{1}{x^1}\leq \frac{1}{x^1}+\frac{1}{x^2} =g(x^1)+g(x^2).
\end{align*}
So, as it is well known, the function $g$ is convex on $]0,\infty[$.
\end{exa}
\begin{exa}
Using a similar calculation to the previous example, one can easily show that the function $g(x)=e^x$ is convex. Indeed, assuming that $x^1\le x^2$ we have
\begin{align*}
\int\limits_0^1\!\int\limits_0^{1-t}e^{sx^1+tx^2}\,ds\,dt\leq
\int\limits_0^1\int\limits_0^{1-t}e^{(s+t)x^2}\ ds\, dt=\frac{1}{2}e^{x^2}\leq \\e^{x^1}+e^{x^2}
\end{align*}
where the right-hand side is equal to $g(x^1)+g(x^2)$.
\end{exa}

\subsection{Proving quasi-convexity with Theorem \ref{T:main_theorem}}

If otherwise not stated, we assume (without losses) throughout the remaining two subsections that $x^1\leq x^2$. 

Let $I\subset\R$ be an open interval, and $g\colon I\to\R$ be a continuous function. If the function $f(x^1,x^2)=\max\{g(x^1),g(x^2)\}$ fulfils inequality \eqref{E:main_result_assumption}, then, according to Theorem \ref{T:main_theorem}, it also fulfils inequality \eqref{E:main_result_statement}, which takes the following shape:

\[
\max\left\{g\left(\frac{x^1+x^2}{2}\right),g\left(\frac{x^1+x^2}{2}\right)\right\}\leq \max\{g(x^1),g(x^2)\},\qquad x^1,x^2\in I,
\]
which means $g$ is Jensen quasi-convex. Using continuity, this implies quasi-convexity of $g$  (see \cite[Remark 1. and Theorem 2.]{NikodemNikodem2009}, \cite[Theorem 2.2.]{Behringer1979} or \cite[Theorem 2. and Corollary 3.]{GilanyiNikodemPales2004}).

So, we have to check the fulfilment of inequality \eqref{E:main_result_assumption}, that is to say:
\[
\frac{1}{2}\int\limits_0^1\int\limits_0^{1-t}\max\{g(sx^1+tx^2),g(tx^1+sx^2)\}\ ds\, dt\leq \max\{g(x^1),g(x^2)\}.
\]
\begin{exa}
Let $g(x)=\log x$. 
\begin{align*}
\frac{1}{2}\int\limits_0^1\int\limits_0^{1-t}\max\{g(sx^1+tx^2),g(tx^1+sx^2)\}\ ds\, dt=\\
\frac{1}{2}\int\limits_0^1\int\limits_0^{1-t}\max\{\log(sx^1+tx^2),\log(tx^1+sx^2)\}\ ds\, dt\leq\\
\frac{1}{2}\int\limits_0^1\int\limits_0^{1-t}\max\{\log(sx^1+tx^1),\log(tx^1+sx^1)\} ds\, dt=\\
\frac{1}{2}\int\limits_0^1\int\limits_0^{1-t} \log((s+t)x^1) ds\, dt=\\-\frac18+\frac14\log x^1\leq\log x^1=\max\{g(x^1),g(x^2)\}.
\end{align*}
So, $g(x)=\log x$ is quasi-convex. Actually, it is concave, hence it is also quasi-concave, which implies that it is quasi-affine.
\end{exa}
\begin{exa}
Let $g(x)=\sqrt{|x|}$. This function is neither convex, nor concave. 
\begin{align*}
\frac{1}{2}\int\limits_0^1\int\limits_0^{1-t}\max\{g(sx^1+tx^2),g(tx^1+sx^2)\}\ ds\, dt=\\
\frac{1}{2}\int\limits_0^1\int\limits_0^{1-t}\max\{\sqrt{|sx^1+tx^2|},\sqrt{|tx^1+sx^2|}\}\ ds\, dt\leq\\
\frac{1}{2}\int\limits_0^1\int\limits_0^{1-t}\sqrt{(s+t)}\sqrt{|x^1|} ds\, dt=\\
\frac{\sqrt{|x^1|}}{5}\leq\sqrt{|x^1|}=\max\{g(x^1),g(x^2)\}.
\end{align*}
So, $g(x)=\sqrt{|x|}$ is quasi-convex.
\end{exa}

\subsection{Proving strong convexity  with Theorem \ref{T:main_theorem}}

The concept of strongly convex functions was introduced in \cite{Poljak1966} by Polyak who proved existence of solutions certain optimization problems. 

Let $I\subset\R$ be an interval. A function $g\colon I\to\R$ is said to be strongly convex with modulus $c$, where $c\geq 0$ is a constant, if
\[
g(tx^1+(1-t)x^2)\leq tg(x^1)+(1-t)g(x^2)-ct(1-t)|x^1-x^2|^2,
\] 
for every $x^1,x^2\in I$ and for for every $t\in[0,1]$. If the previous inequality is fulfilled with fixed $t=\frac12$, then $g$ is called strongly midconvex (see e.g. \cite{AzocarGimenezNikodemEtAl2011}).

In \cite[Lemma 2.1.]{NikodemPales2011} the authors gave the following nice characterization of strongly convex functions.
\begin{thm}\label{T:characterization_of_strong_convexity}
With the previous notations, a function $g$ is strongly convex(strongly midconvex) with modulus $c$ if and only if the function
\[
\varphi=g-c|\cdot|^2
\]
is convex(midconvex).
\end{thm}
Furthermore, a Bernstein-Doetsch type result is proved in \cite[Corollary 2.2.]{AzocarGimenezNikodemEtAl2011}.
\begin{thm}\label{T:BD_strong_convexity}
A continuous function is strongly convex with modulus $c$ if and only if it is strongly midconvex with modulus $c$.
\end{thm}

Let $g\colon I\to\R$ be a continuous strongly convex function with modulus $c$. Then, because of Theorem \ref{T:characterization_of_strong_convexity} there is a convex function $\varphi\colon I\to\R$ such that $\varphi=g-c|\cdot|^2$. Following the trail of thoughts at the beginning of subsection \ref{S:proving_convexity} with $\varphi$ we get that the function $f(x^1,x^2)=\varphi(x^1)+\varphi(x^2)$ fulfils inequality \eqref{E:main_result_assumption}, then, according to Theorem \ref{T:main_theorem}, it also fulfils inequality \eqref{E:main_result_statement}, which takes the following shape:
\[
2\varphi\left(\frac{x^1+x^2}{2}\right)\leq \varphi(x^1)+\varphi(x^2),\qquad x^1,x^2\in I,
\]
which can be transformed into
\[
g\left(\frac{x^1+x^2}{2}\right)-\frac{c}{4}|x^1+x^2|^2\leq \frac{g(x^1)+g(x^2)}{2}-\frac{c}{2}(|x^1|^2+|x^2|^2)
\]
Because
\[
-\frac{c}{2}(|x^1|^2+|x^2|^2)+\frac{c}{4}|x^1+x^2|^2=-\frac{c}{4}|x^1-x^2|^2,
\]
we get that $g$ is a continuous, strongly midconvex function with modulus $c$. Theorem \ref{T:BD_strong_convexity} implies the strong convexity of $g$.

So, we have to check the fulfilment of inequality \eqref{E:main_result_assumption}, that is to say:
\begin{align*}
\frac{1}{2}\int\limits_0^1\int\limits_0^{1-t}\big(g(sx^1+tx^2)-\frac{c}{4}|sx^1+tx^2|^2+\\g(tx^1+sx^2)-\frac{c}{4}|tx^1+sx^2|^2\big)\ ds\, dt\leq\\ g(x^1)-\frac{c}{4}|x^1|^2+g(x^2)-\frac{c}{4}|x^2|^2.
\end{align*}
More precisely, just as in the case of convexity (Section 5.1) it is equivalent to the following inequality:
\begin{align*}
\int\limits_0^1\int\limits_0^{1-t}\big(g(sx^1+tx^2)-\frac{c}{4}|sx^1+tx^2|^2+\big)\ ds\, dt\leq\\ g(x^1)-\frac{c}{4}|x^1|^2+g(x^2)-\frac{c}{4}|x^2|^2.
\end{align*}
Since the integral of the modulus part is equal to 
\begin{align*}
-\frac{c}{4}\int\limits_0^1\int\limits_0^{1-t}\big(|sx^1+tx^2|^2\big)\ ds\, dt\\= -\frac{c}{24}\left((x^1)^2+x^1x^2+(x^2)^2\right),
\end{align*}
the above-mentioned inequality can be rewitten as
\begin{align}\label{E:strongconvex}
\int\limits_0^1\int\limits_0^{1-t}\big(g(sx^1+tx^2)\big)\ ds\, dt\leq\nonumber\\ g(x^1)+g(x^2)+\\\frac{c}{48}(x^1x^2-11(x^1)^2-11(x^2)^2).\nonumber
\end{align}

\begin{exa}
Let $g(x)=\frac{1}{x}$ and $0<a<b<\infty$. We intend to prove that $g$ is strongly convex with modulus $c=\frac{1}{b^3}$. For this, we can perform the following estimates concerning the right hand side of \eqref{E:strongconvex}.
\begin{align*}
\int\limits_0^1\int\limits_0^{1-t}\left(\frac{1}{sx^1+tx^2}\right)\ ds\, dt\leq \int\limits_0^1\int\limits_0^{1-t}\left(\frac{1}{sx^1+tx^1}\right)\ ds\, dt=\\
\frac{1}{x^1}\int\limits_0^1\int\limits_0^{1-t}\frac{1}{s+t}\ ds\, dt=\frac{1}{x^1}.
\end{align*}
For the estimation of the left hand side of \eqref{E:strongconvex} we can write the followings.
\begin{align*}
\frac{1}{x^1}+\frac{1}{x^2}+\frac{1}{b^348}(x^1x^2-11(x^1)^2-11(x^2)^2)\geq\\ \frac{1}{x^1}+\frac{1}{x^2}+\frac{1}{b^348}(a^2-22b^2)\geq  \frac{1}{x^1}+\frac{1}{b}+\frac{a^2}{48b^3}-\frac{1}{2b}\geq\frac{1}{x^1}.
\end{align*}
Which implies that $g$ is really strong convex with modulus $\frac{1}{b^3}$ on the interval $[a,b]$.
\end{exa}

\section{Acknowledgement}

This research is supported by the Hungarian Scientific Research Fund (OTKA) Grant PD 12487.

%------------------------------------------------
%------------------------------------------------

%\bibliographystyle{plain}
%\bibliography{Burai_bib}

%------------------------------------------------

\end{document}